\documentclass{amsart}
\usepackage[dvipsnames,usenames]{xcolor}
\usepackage{amsmath,amssymb,amsthm, tikz-cd}
\usepackage{stmaryrd}
\usepackage{graphicx}
\usepackage{enumitem}
\usepackage[margin=1in]{geometry}
\usepackage{dutchcal}
\usepackage{colonequals}
\usepackage{diagbox}
\usepackage{tabularray}
\usepackage{comment}
\usepackage{hyperref}
\usepackage{mabliautoref}
\usepackage[normalem]{ulem}
\input xy
\xyoption{all}
\input{kmacros3.sty}
\setlist[enumerate]{label = (\roman*), topsep=0pt, itemsep=0pt}

\newlist{enuroman}{enumerate}{1}
\setlist[enuroman]{
	label = (\roman*),
	topsep = 0pt,
	labelindent = \parindent}
	
\newlist{enualph}{enumerate}{2}
\setlist[enualph]{
	label = (\alph*),
	topsep = 0pt,
	labelindent = \parindent}

\newlist{enuarab}{enumerate}{3}
\setlist[enuarab]{
	label = (\arabic*),
	topsep = 0pt,
	labelindent = \parindent}

\setlist[itemize]{label = $\diamond$, topsep=0pt, itemsep=0pt, labelindent=0pt}

\DeclareRobustCommand\iff{\;\Longleftrightarrow\;}

\newcommand{\sq}{\subseteq}

\newcommand{\arrow}{arrow}
\newcommand{\F}{{\bf F}}

\newcommand{\fm}{\mathfrak{m}}

\newcommand{\ceq}{\colonequals}

\DeclareMathOperator{\lct}{lct}

\renewcommand{\phi}{\varphi}
\renewcommand{\to}{\longrightarrow}

\newcommand{\llb}{[ \! [}
\newcommand{\rrb}{] \! ]}

\DeclareMathOperator{\ppt}{ppt}

\renewcommand{\mod}{\text{ mod }}

\begin{document}

\title[Plus-pure thresholds of some cusps in mixed characteristic]{Plus-pure thresholds of some cusp-like singularities \\ in mixed characteristic}
\author[Cai]{Hanlin Cai}
\address{Department of Mathematics, Columbia University, 2990 Broadway, Mathematics Hall, Rm 307A, 
New York New York 10027}
\email{hc3589@columbia.edu}

\author[Pande]{Suchitra Pande}
\address{Department of Mathematics, University of Utah, Salt Lake City, UT 84112, USA}
\email{suchitra.pande@utah.edu}

\author[Quinlan-Gallego]{Eamon Quinlan-Gallego}
\address{Department of Mathematics, University of Utah, Salt Lake City, UT 84112, USA}
\email{eamon.quinlan@utah.edu}

\author[Schwede]{Karl Schwede}
\address{Department of Mathematics, University of Utah, Salt Lake City, UT 84112, USA}
\email{schwede@math.utah.edu}

\author[Tucker]{Kevin Tucker}
\address{Department of Mathematics, University of Illinois at Chicago, Chicago, IL, USA}
\email{kftucker@uic.edu}
\maketitle

\begin{abstract}
	Log-canonical and $F$-pure thresholds of pairs in equal characteristic admit an analog in the recent theory of singularities in mixed characteristic, which is known as the plus-pure threshold. In this paper we study plus-pure thresholds for singularities of the form $p^a + x^b \in \bZ_p \llb x \rrb$, showing that in a number of cases this plus-pure threshold agrees with the $F$-pure threshold of the singularity $t^a + x^b \in \F_p \llb t, x \rrb$. We also discuss a few other sporadic examples.
\end{abstract}

\section{Introduction}

Let $X$ be a smooth algebraic variety over ${\bf C}$, let $x \in X$ be a closed point and $D = (f = 0)$ be a prime divisor containing $x$. If $R \ceq \widehat{\cO_{X, x}}$ denotes the completion of the local ring at $x$, the log-canonical threshold $\lct(R, f)$ is the threshold in $t$ at which the pair $(X, tD)$ becomes not log canonical at $x$; in other words, it is the first jumping number of the multiplier ideal $\mathcal J(R, f^t)$. This invariant has been used prominently in birational algebraic geometry and the minimal model program; see  \cite{KollarSingularitiesOfPairs,LazarsfeldPositivity2} for some applications. 

Log canonical pairs are closely related to $F$-pure pairs in characteristic $p > 0$ \cite{HaraWatanabeFRegFPure,MustataSrinivasOrdinary,BhattSchwedeTakagiweakordinaryconjectureandFsingularity}. Inspired by this connection, Takagi and Watanabe defined the $F$-pure threshold $\fpt(R, f)$ as the threshold in $t$ at which $(R, f^t)$ becomes not $F$-pure \cite{TakagiWatanabeFPureThresh}, \cf \cite{MustataTakagiWatanabeFThresholdsAndBernsteinSato,HunekeMustataTakagiWatanabeFThresholdsTightClosureIntClosureMultBounds}.  When $R$ is regular, this is also the first jumping number of the test ideal $\tau(R, f^t)$, a characteristic-$p$ analog of the multiplier ideal.  Note that there has been a great deal of effort directed towards computing the $F$-pure thresholds of specific (classes of) polynomials \cite{ShibutaTakagiLCThresholds,HernandezFInvariantsOfDiagonalHyp,HernandezFPureThresholdOfBinomial,BhattSingh.FPTOfCalabiYau,MillerSinghVarbaro.FPTOfDeterminantal,HernandezNunezBetancourtWittZhang.FPTOfHomogeneous,HernandezTeixeira.FThresholdFunctionSyzygyGapFractals,DeStefaniNunezBetancourt.FThresholdsOfGradedRings,Muller.FPTOfQuasiHomogeneous,GiladElementaryComputationOfFPTOfEllipticCurve,KadyrsizovaKenkelPageSinghSmithVraciuWitt.LowerBoundsExtremal,SmithVraciuValuesOfFPTForHomog}. 

In recent years, building upon the theory of Scholze's perfectoid algebras \cite{ScholzePerfectoidspaces} and the work of Andr\'e, Bhatt, Gabber, Heitmann-Ma and others  \cite{AndreDirectsummandconjecture,BhattDirectsummandandDerivedvariant,AndreWeaklyFunctorialBigCM,GabberMSRINotes,HeitmannMaBigCohenMacaulayAlgebraVanishingofTor,BhattAbsoluteIntegralClosure}, the development of a theory of singularities in mixed characteristic is now underway \cite{MaSchwedePerfectoidTestideal,MaSchwedeSingularitiesMixedCharBCM,DattaTuckerOpenness,MaSchwedeTuckerWaldronWitaszekAdjoint,BMPSTWW1,RobinsonToricBCM,HaconLamarcheSchwede.GlobalGenOfTestIdeals,MurayamaSymbolicTestIdeal,CaiLeeMaSchwedeTuckerPerfectoidHilbertKunz,BMPSTWW-RH,BMPSTWW-PerfectoidPure,RodriguezBCMThresholds}. This theory provides an analog for the multiplier and test ideals.  The first jumping number---known as the {\em plus-pure threshold}---therefore provides an analog of the log-canonical and $F$-pure thresholds. A succinct definition is as follows: if $(R, \fm)$ is a regular complete local domain of mixed characteristic, $0 \neq f \in R$ is an element, and $t \in \bQ_{>0}$ is a positive rational number, we denote by $f^t: R \to R^+$ the unique $R$-linear map sending $1 \in R$ to a choice of $f^t \in R^+$; then the plus-pure threshold of $(R,f)$ is given by
$$\ppt(R, f) \ceq \sup \{ t \in \bQ_{> 0} \ | \ f^t:  R \to  R^+ \text{ splits }\}, $$
where $R^+$ is the absolute integral closure of $R$ (see \autoref{lem.EquivPurityForCompletionIdealContainment} for other characterizations).

Except for some toric-type singularities, see for instance \cite{RobinsonToricBCM}, essentially no examples of plus-pure thresholds have been calculated (although some upper and lower bounds were known, and sometimes they coincide, see below). This paper begins this exploration by computing the plus-pure thresholds of some cusp-like singularities $f \ceq p^a + x^b \in \bZ_p \llb x \rrb$, where $\Z_p$ is the ring of $p$-adic integers.  Indeed, unlike characteristic $p > 0$, one cannot simply use Frobenius and so the computations become much more subtle.

Before we discuss our work, we remark that something was known about the examples $f \ceq p^2 + x^3 \in \Z_p \llb x \rrb$ and $f \ceq p^3 + x^2 \in \Z_p \llb x \rrb$. For either choice, by \cite[Proposition 4.17, Lemma 4.18]{BMPSTWW1} or \cite[Theorem 6.21]{MaSchwedeSingularitiesMixedCharBCM} and \cite{BhattAbsoluteIntegralClosure}, {if the map $\Z_p \llb x \rrb \to \Z_p \llb x \rrb^+$ given by $1 \mapsto f^t$ were to split for some $t \in \Q_{>0}$, the pair $(\Z_p \llb x \rrb, f^t)$ would be log-terminal, then a blowup computation shows that $t < 5/6$ and therefore $\ppt(\Z_p \llb x \rrb, f) \leq 5/6$}. On the other hand, the $F$-pure threshold of the related polynomial $f_0 \ceq y^2 + z^3 \in \F_p \llb y, z \rrb$ provides a lower bound for $\ppt(\Z_p \llb x \rrb, f)$; indeed, combining \cite[Example 7.10]{MaSchwedeTuckerWaldronWitaszekAdjoint} with \cite{BhattAbsoluteIntegralClosure}, we see that $f^t: \Z_p \llb x \rrb \to \Z_p \llb x \rrb^+$ splits whenever $(\F_p \llb y, z \rrb, f_0^t)$ is $F$-regular. We conclude that
$$\fpt \big( \F_p \llb y, z \rrb, f_0 \big) \leq \ppt \big( \Z_p \llb x \rrb, f \big) \leq \frac{5}{6},$$
and we recall that the lower bound $\fpt(\F_p \llb y, z \rrb, f_0)$ is known to be as follows (see \cite[Example 4.3]{MustataTakagiWatanabeFThresholdsAndBernsteinSato}):

$$\fpt \big( \F_p \llb y, z \rrb , f_0 \big) = 
\begin{cases} 
	1/2 \text{ if } p = 2 \\ 
	2/3 \text{ if } p = 3 \\
	(5p - 1)/6p \text{ if } p \equiv 5 \mod 6 \\
	5/6 \text{ if } p \equiv 1 \mod 6.
\end{cases}$$
In particular, it was known that $\ppt( \Z_p \llb x \rrb, f) = 5/6$ when $p \equiv 1 \text{ mod } 6$.

As our first contribution, we show that analogous upper and lower bounds hold for any cusp-like polynomial $f = p^a + x^b \in \bZ_p \llb x \rrb$; this is verified in detail in \autoref{prop.fpt<=+pt<=lct} by utilizing some of the main results of \cite{MaSchwedeSingularitiesMixedCharBCM,MaSchwedeTuckerWaldronWitaszekAdjoint,BMPSTWW1}.

Our main goal is then to show that, in many cases, one has $\ppt(R, f) = \fpt(R_0, f_0)$; in other words, that one has equality on the lower bound. For example, we prove that this is the case for the polynomials $f = p^2 + x^3$ and $f = p^3 + x^2$ discussed above. Note that, by \autoref{prop.fpt<=+pt<=lct} as mentioned above, one has $\ppt(f) = {1 \over a} + {1 \over b}$ whenever $\fpt(R_0, f_0) = {1 \over a} + {1 \over b}$, and some of the remaining cases are handled by the theorem below.

\begin{mainthm*}[{\autoref{cor-ppt-test}}] \label{mainthm-ppt-test}
	Fix a prime $p > 0$, let $R \ceq \Z_p\llbracket x \rrbracket$ be a power series ring in one variable over the $p$-adic integers, and let $R_0 \ceq \F_p \llbracket y, x \rrbracket$ be a power series ring in two variables over $\F_p$. Fix integers $a, b \geq 2$ and consider the polynomials $f \ceq p^a + x^b \in R$ and $f_0 \ceq y^a + x^b \in R_0$.

	Assume that $\fpt(R_0, f_0) \neq {1 \over a} + {1 \over b}$, and fix an integer $e \geq 1$ such that $p^e \fpt(R_0, f_0) \in \Z$. If we have 
	$$\fpt(R_0, f_0) \geq \frac{1}{a} + \frac{1}{b} - \frac{1}{a p^e},$$
	then $\ppt(R, f) = \fpt(R_0, f_0)$.
\end{mainthm*}

Note that an algorithm for the computation of $\fpt(R_0, f_0)$ is provided in \cite{HernandezFPureThresholdOfBinomial,HernandezFInvariantsOfDiagonalHyp}, \cf \cite{ShibutaTakagiLCThresholds} (we briefly discuss this algorithm in Section 3). It follows from some of the references above that, unless one has $\fpt(R_0, f_0) = {1 \over a} + {1 \over b}$, the denominator of $\fpt(R_0, f_0)$ is a power of $p$, and therefore one can always choose an $e \geq 1$ as in the statement of the theorem.

Since the $F$-pure threshold $\fpt(R_0, f_0)$ can be computed following the algorithm mentioned above, our result gives an effective method for finding values of $p,a,b$ for which we have $\ppt(p^a + x^b) = \fpt(y^a + x^b)$. Using the slightly stronger version of the Main Theorem given in \autoref{cor-ppt-test}, as well as the discussion thereafter, we know that $\ppt(R, p^a + x^b) = \fpt(R_0, y^a + x^b)$ for the primes given in Table \ref{table.intro} (we also prove some more cases in \autoref{sec.sporadic}).

\begin{table}[h]
    \centering
    \caption{Values of $p$ for which we know $\ppt(p^a + x^b) = \fpt(y^a + x^b)$.}
    \label{table.intro}    
\begin{tblr}{|Q[c,m]|Q[c,m]|Q[c,m]|Q[c,m]|Q[c,m]|} 
\hline 
\diagbox[width=3em]{$a$}{$b$} & 2 & 3 & 4 & 5 \\ 
\hline 
2 & All $p$ & All $p$ & All $p$ & All $p$ \\ 
\hline
3 & All $p \neq 2$ & All $p \neq 3$ & {All $p \neq 3$ with \\ $p \not\equiv 11 \text{ mod }12$} &  {All $p \neq 3, 5$ \\ with $p \not\equiv 14 \text{ mod }15$}\\ 
\hline
4 & All $p$ & {All $p \neq 3$ with \\ $p \not \equiv 11 \text{ mod }12$} & {All $p \neq 2$ with \\ $p \not \equiv 3 \text{ mod } 4$} & {All $p \neq 2$ with \\ $p \not \equiv 3, 19 \text{ mod }20$ } \\ 
\hline
5 & {All $p \neq 2,5$ with \\ $p \not \equiv 7,9 \text{ mod } 10$} & {All $p \neq 3,5$ with \\ $p \not \equiv 8, 14 \text{ mod }15$} & {All $p \neq 2$ with \\ $p \not  \equiv 3,19 \text{ mod }20$} & {All $p \neq 5$ with \\ $p \not \equiv  2, 4 \text{ mod }5$}\\
\hline
\end{tblr}
\end{table}

One can also fix the residue characteristic $p$, and ask for which values of $a, b$ our main result applies.  For $p = 3$, we display these in a scatter plot in Figure \ref{scatterplot}. While the condition is not dense, one can see there are numerous values of $(a,b)$ where it holds. The plots for other primes look similar.

\begin{figure}[h]
    \label{scatterplot}    
    \includegraphics[scale=0.135]{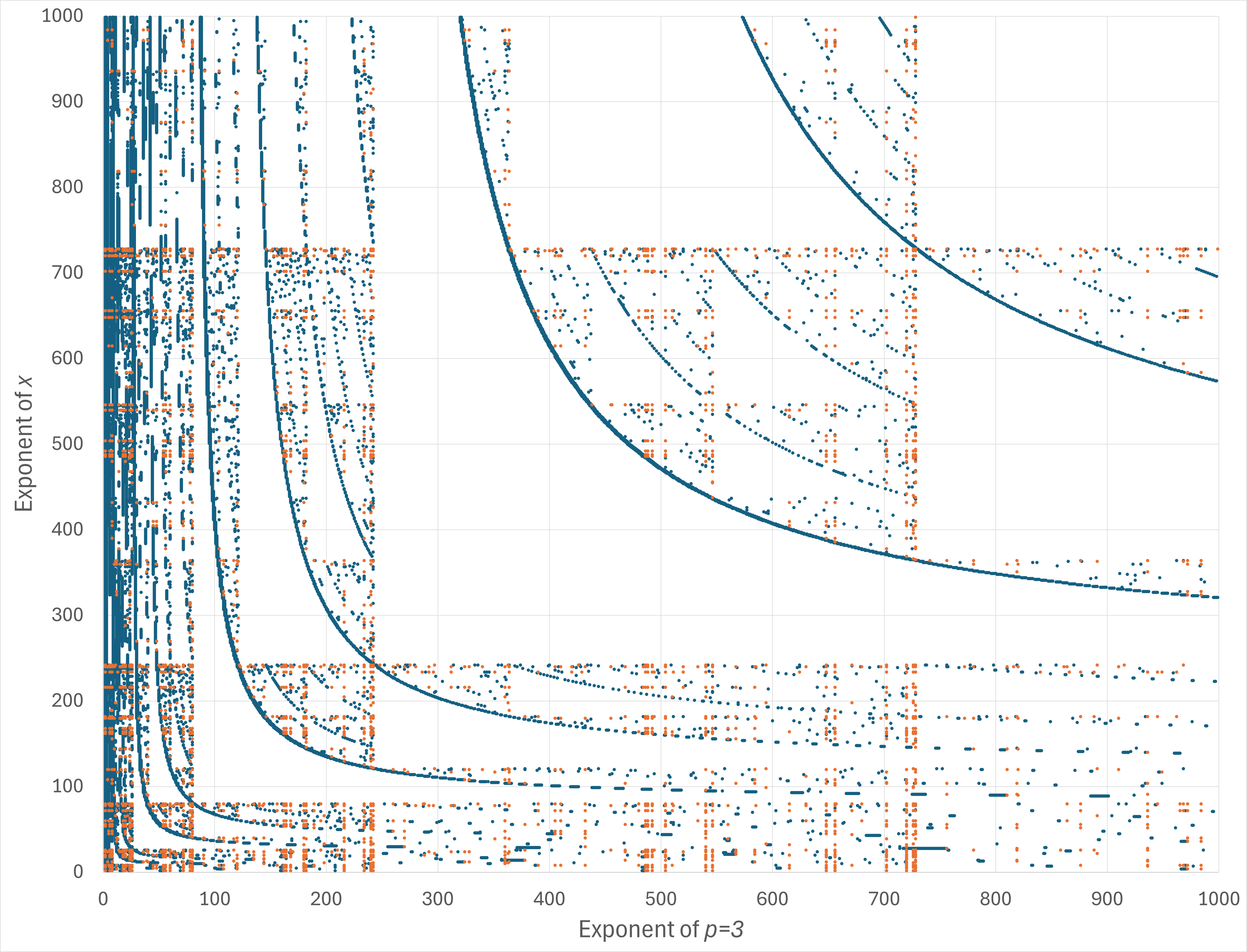}
    \caption{Values of $(a,b)$ where we know $\ppt(p^a + x^b)$ when $p=3$.  The orange dots correspond to the $a$ and $b$ values where $\fpt(f_0) = {1 \over a} + {1 \over b}$.  The remaining dots are consequences of \autoref{cor-ppt-test}.}
\end{figure}

The code to generate this data set is included as an ancillary file to the arXiv version of this document and is also available on Schwede's website.  
\begin{center}
    \url{https://www.math.utah.edu/~schwede/M2/MakeTable-PPT.m2}
\end{center}
The code is written in Macaulay2 \cite{M2} and uses the {\tt FrobeniusThresholds} package \cite{HernandezSchwedeTeixeiraWittFrobeniusThresholds}.

In \autoref{sec.sporadic}, we study more sporadic examples not covered in our main theorem.  First, we show that if $f = x^2 + y^3 \in \bZ_2\llbracket x,y\rrbracket = R$, then $\ppt(R, f) > 1/2 = \fpt(x^2 + y^3 \in \bF_2\llbracket x,y\rrbracket)$, see \autoref{cor.CuspNoPChar2}. In \autoref{lem.3Lines}, we study the $\ppt$ of the analog of ``three lines'', $px(x-p) \in \bZ_p\llbracket x \rrbracket$, showing it agrees with the $\fpt$ of the analogous three lines. Finally, in \autoref{prop.2ax2} we show that the plus-pure threshold of $2^a + x^2 \in \Z_2 \llb x \rrb$ is $1/2$ for all integers $a \geq 2$.

\subsection*{Acknowledgments}  The authors thank Linquan Ma for valuable discussions and comments on a previous draft.  In particular, Linquan Ma pointed out that \autoref{thm.CuspNoPChar2} holds for arbitrary big Cohen-Macaulay algebras (and not just $R^+$) which then shows \autoref{cor.CuspNoPChar2}.  The authors thank the anonymous referee as well as Marta Benozzo, Vignesh Jagathese, Vaibhav Pandey, and Pedro Ramirez-Moreno for valuable comments on previous drafts. Special thanks to Marta Benozzo for catching an error in a prior version of Proof \#2 in \autoref{lem.FPTCoverEqualityWithRamification}.
This material is partly based upon work supported by the National Science Foundation under Grant No. DMS-1928930 and by the Alfred P. Sloan Foundation under grant G-2021-16778, while the authors were visiting the Simons Laufer Mathematical Sciences Institute (formerly MSRI) in Berkeley, California, during the Spring 2024 semester.  
Hanlin Cai was partially supported by NSF FRG Grant \#1952522. Pande was partially supported by NSF Grant \#2101075. Quinlan-Gallego was partially supported by an NSF Postdoctoral Research Fellowship \#2203065.  
Schwede was partially supported by NSF Grant \#2101800 and by NSF FRG Grant \#1952522.
Tucker was partially supported by NSF Grant \#2200716.

\section{Preliminaries}

As usual, a prime number $p > 0$ is fixed. 

\subsection{Plus pure thresholds}

Recall a commutative integral domain $S$ is said to be {\it absolutely integrally closed} if every monic polynomial in $S[T]$ splits into a product of linear factors or, equivalently, if every monic polynomial with coefficients in $S$ has a solution in $S$. Given an arbitrary integral domain $R$ with fraction field $K$, after fixing an algebraic closure $\overline K$ of $K$, the integral closure $R^+$ of $R$ in $\overline K$ is an absolutely integrally closed domain, called the {\it absolute integral closure of $R$}.

If $S$ is an absolutely integrally closed domain and $z \in S$ is an element, given a nonnegative rational number $\alpha = n / m$ with $n, m \in \Z$, we let $z^\alpha \in S$ denote a choice of element for which $(z^\alpha)^m = z^n$. Note that up to units such a choice of $z^\alpha$ is unique and independent of the representation $\alpha = n / m$. We will be sloppy about the nonuniqueness: in general we will only make statements about the membership of $z^\alpha$ in some ideal of $R^+$ or statements about ideals in which $z^\alpha$ is a generator.  Such statements are, of course, independent of the choice of $z^\alpha$.

\begin{definition}
    \label{def.PPT}
    Suppose $(S, \fm)$ is a complete Noetherian local domain of residue characteristic $p > 0$.  We suppose further that $0 \neq f \in \fm$.  In this case we define the \emph{plus-pure-threshold} of $f$ to be
    \[
        \ppt(f) := \sup\big\{ \alpha \in \bQ_{\geq 0} \;\big|\; S \xrightarrow{1 \mapsto f^\alpha} S^+ \; \text{ is pure} \big\}.
    \]    
    In this paper, we only study this condition in the case where $S$ is regular.
\end{definition}

We write down some equivalent ways to check the purity of the map $S \xrightarrow{1 \mapsto f^{\alpha}} S^+$.

\begin{lemma}
    \label{lem.EquivPurityForCompletionIdealContainment}
    With notation as in \autoref{def.PPT} suppose $\alpha \in \bQ_{\geq 0}$, and $f^{\alpha} \in S^+$ is fixed. Then the following is equivalent.
    \begin{enumerate}  
        \item The map $S \xrightarrow{1 \mapsto f^{\alpha}} S^+$ is pure (equivalently splits). \label{lem.EquivPurityForCompletionIdealContainment.PurityNoCompletion} 
        \item If $E$ is an injective hull of $k := S/\fm$, then the induced map $S \otimes_S E \xrightarrow{1 \mapsto f^{\alpha}} S^+ \otimes_S E$ is injective. \label{lem.EquivPurityForCompletionIdealContainment.PurityViaInjectiveHull} 
        \item The map $S \xrightarrow{1 \mapsto f^{\alpha}} \widehat{S^+}$ is pure (equivalently splits) where $\widehat{S^+}$ is the $p$-adic (or $\fm$-adic) completion of $S^+$. 
        \label{lem.EquivPurityForCompletionIdealContainment.PurityWithCompletion}
        \item For each finite extension $S \subseteq S' \subseteq S^+$ containing $f^{\alpha}$ we have that 
        \[
            S \xrightarrow{1 \mapsto f^\alpha} S'
        \]
        splits.\label{lem.EquivPurityForCompletionIdealContainment.SplitByFiniteCovers}    
        \item $\tau_+(S, f^{\alpha}) = S$ where $\tau_+(S, f^{\alpha}) = \tau_+(S, \alpha \Div(f))$ is the plus-test-ideal denoted as $\tau_{\widehat{S^+}}(S, f^{\alpha})$ in \cite{MaSchwedeSingularitiesMixedCharBCM}, see also \cite[Definition 4.16]{BMPSTWW1}, \cite[Section 4]{HaconLamarcheSchwede.GlobalGenOfTestIdeals} (setting $X = \Spec S$), or \cite[Section 8]{BMPSTWW-RH}.  \label{lem.EquivPurityForCompletionIdealContainment.TestIdealVariant}
    \end{enumerate}
    If additionally we suppose that $S$ is regular, then these conditions are equivalent to:
    \begin{enumerate}
    \setcounter{enumi}{5}
        \item $f^{\alpha} \notin \fm S^+$. \label{lem.EquivPurityForCompletionIdealContainment.IdealContainmentNoCompletion}
        \item $f^{\alpha} \notin \fm \widehat{S^+}$ where $\widehat{S^+}$ is the $p$-adic (or $\fm$-adic) completion of $S^+$        \label{lem.EquivPurityForCompletionIdealContainment.IdealContainmentWithompletion}        
    \end{enumerate}
\end{lemma}
\begin{proof}
    In parts \ref{lem.EquivPurityForCompletionIdealContainment.PurityNoCompletion} and \ref{lem.EquivPurityForCompletionIdealContainment.PurityWithCompletion} we see that the ``equivalently splits'' statements are harmless as $S$ is complete so that pure and split maps from $S$ are equivalent, see \cite[Lemma 1.2]{FedderFPureRat}. The equivalence of \ref{lem.EquivPurityForCompletionIdealContainment.PurityNoCompletion} and \ref{lem.EquivPurityForCompletionIdealContainment.PurityViaInjectiveHull} follows easily from Matlis duality, see \cite[Lemma (2.1)(e)]{HochsterHunekeApplicationsofBigCM}. Then the equivalence of \ref{lem.EquivPurityForCompletionIdealContainment.PurityNoCompletion} and \ref{lem.EquivPurityForCompletionIdealContainment.PurityWithCompletion} follows from the fact that $\widehat{S^+}\otimes_SE\simeq S^+\otimes_S E$ so we win by using \ref{lem.EquivPurityForCompletionIdealContainment.PurityViaInjectiveHull}. Certainly \ref{lem.EquivPurityForCompletionIdealContainment.SplitByFiniteCovers} and \ref{lem.EquivPurityForCompletionIdealContainment.PurityNoCompletion} are equivalent since a filtered colimit of split maps is pure and a pure finite map of Noetherian rings is split, see \cite[\href{https://stacks.math.columbia.edu/tag/058H}{Tag 058H}]{stacks-project}.

        Unraveling the definition of $\tau_+(S, f^{\alpha})$ easily implies that \ref{lem.EquivPurityForCompletionIdealContainment.TestIdealVariant} is equivalent to many of the other definitions above. More precisely, since $\tau_+(S, f^{\alpha})$ is the Matlis dual of the image of $S\otimes_S E\to S^+\otimes_S E$, $\tau_+(S, f^{\alpha})=S$ if and only if $S\otimes_S E\to S^+\otimes_S E $ is injective.

    Clearly, \ref{lem.EquivPurityForCompletionIdealContainment.IdealContainmentNoCompletion} and \ref{lem.EquivPurityForCompletionIdealContainment.IdealContainmentWithompletion} are equivalent since $p\in \mathfrak{m}$, see \cite[\href{https://stacks.math.columbia.edu/tag/05GG}{Tag 05GG}]{stacks-project}. Therefore, to complete the proof, it suffices to prove the equivalence of \ref{lem.EquivPurityForCompletionIdealContainment.IdealContainmentWithompletion} and \ref{lem.EquivPurityForCompletionIdealContainment.PurityViaInjectiveHull}. If $f^\alpha\in \mathfrak{m}\widehat{S^+}$, then $S\to \widehat{S^+}$ is not pure by tensoring with $k$. It is for the converse that we finally need regularity.  Assume that $f^\alpha\notin \mathfrak{m}\widehat{S^+}$. Since $S$ is regular, $\widehat{S^+}$ is big Cohen-Macaulay and so is faithfully flat over $S$ by \cite[Page 77]{HochsterHunekeInfiniteIntegralExtensionsAndBigCM} or \cite[Lemma 2.9]{BhattAbsoluteIntegralClosure} (it is worth noting that if $S \to \widehat{S^+}$ is faithfully flat, then $S$ is regular by \cite{BhattIyengarMaRegularRingsPerfectoid}). Hence the injectivity of $S/\m =: k \to E$ extends to the injectivity of $k\otimes_S \widehat{S^+} \to E\otimes_S \widehat{S^+}$, giving the injectivity of $k \to E\otimes_S \widehat{S^+}$. By considering the socle, we have that $E\to E\otimes_S \widehat{S^+}$ is injective as desired.
\end{proof}

\begin{remark}
In a previous version of this article, we denoted $\ppt(f)$ by $+\text{pt}(f)$.  Because including the $+$ in the notation caused confusion, especially when writing on a board, we have switched to $\ppt$.  

    The closely related \emph{perfectoid pure threshold of $p$}, unfortunately also denoted by $\ppt(R, p)$, was defined as 
\[
\sup\big\{ a/p^e \in \bQ_{\geq 0}\;\big|\; S \xrightarrow{1 \mapsto \varpi^a} B \,\text{ is pure for some perfectoid $R$-algebra $B$ with $\varpi$ a $p^e$th root of $p$}\big\},
\] 
see \cite{yoshikawa2025computationmethodperfectoidpurity}.  

As that definition can clearly be generalized to any element of $S$, there is potential for confusion.  However, we believe that two notions are the same at least in the context of this paper where $S = W(k)\llbracket x_2, \dots, x_d\rrbracket$.  The key ingredient is a slight modification of \cite[Lemma 5.1.6]{CaiLeeMaSchwedeTuckerPerfectoidHilbertKunz} which lets one pass from $S^+$ to perfectoid algebras that can be used to test the perfectoid threshold up to perturbation by $f^\epsilon$ for $1 \gg \epsilon > 0$ which will not impact the supremum in the threshold.   For a discussion of the choice of such perfectoid algebras, see \cite[Section 2.3]{yoshikawa2025computationmethodperfectoidpurity}, and also \cite[Lemma 4.23]{BMPSTWW-PerfectoidPure} and \cite[Proposition 3.5]{Fayolle.PerfectoidCenters}
\end{remark}

\subsection{Ideal containments}

\begin{lemma} \label{lemma-key}
	Let $S$ be an absolutely integrally closed domain, let $z, y, y_1, \dots , y_s \in S$ be elements, let $e \geq 1$ be an integer and $\epsilon \in \Q$ be a rational number. 
	\begin{enuroman}
	\item If $\epsilon \in (0, \frac{p}{p-1}]$ then we have
	$$z \in (p^\epsilon, y) \iff z^{1/p^e} \in (p^{\epsilon / p^e}, y^{1/p^e} ).$$
	\item If $\epsilon \in (0, 1]$, we have
	$$z \in (p^\epsilon, y_1, \dots , y_s) \iff z^{1/p^e} \in (p^{\epsilon / p^e}, y_1^{1 / p^e}, \dots,  y_s^{1 / p^e}).$$
	\item If $\epsilon \in (0, \frac{p}{p-1}]$ then we have
	$$z \in (p^\epsilon, y) \implies z^\gamma \in \bigg( p^{\epsilon \alpha} \ y^{\beta} \ \bigg| \ {{\alpha, \beta \in \frac{1}{p^e} \Z \cap [0,1)} \atop {\alpha + \beta = \gamma  }} \bigg)$$

	\item If $\epsilon \in (0,1]$ and $\gamma \in \frac{1}{p^e} \Z \cap (0,1)$ then we have
	$$z \in (p^\epsilon, y_1, \dots , y_s) \implies z^\gamma \in \bigg( p^{\epsilon \alpha_0} y_1^{\alpha_1} \cdots y_s^{\alpha_s} \ \bigg| \ {{\alpha_i \in \frac{1}{p^e} \Z \cap [0,1)} \atop {\alpha_0 + \cdots + \alpha_s = \gamma  }} \bigg)$$
	\end{enuroman}
\end{lemma}

These for $\epsilon \leq 1$ results can be deduced from \cite[Proposition 5.5]{ScholzePerfectoidspaces} as the $p$-adic completion of an absolutely integrally closed domain is perfectoid.  We provide a direct and down-to-earth proof for the convenience of the reader.

\begin{proof}
	We start with statement (i), and we observe that it is enough to prove it in the case $e = 1$. For the $(\Longrightarrow)$ direction, after replacing $y$ with a multiple we may assume that $z = a p^\epsilon + y$ for some $a \in S$, and we look for elements $s \in S$ satisfying $z^{1/p} = s p^{\epsilon / p} + y^{1/p}$ or, in other words,
	\begin{align*}
		ap^\epsilon + y & = (s p^{\epsilon / p} + y^{1/p})^p \\
			& = s^p p^{\epsilon} + p^{1 + (\epsilon / p)} \big( \cdots \big) + y.
	\end{align*}
	The assumption on $\epsilon$ ensures that $\epsilon \leq 1 + (\epsilon / p)$, and therefore after canceling the $y$ terms one can divide the above equation by $p^\epsilon$, which ensures that the resulting equation on $s$ is monic. Since $S$ is absolutely integrally closed by assumption, there are elements $s \in S$ satisfying the required equation.

	For the $(\Longleftarrow)$ direction, after replacing $y$ with a multiple we may assume that $z^{1/p} = s p^{\epsilon / p} + y^{1/p}$ for some $s \in S$, and the equation above gives $z \in (p^\epsilon, y)$. 

e	As before, it suffices to prove statement (ii) in the case $e = 1$, and we begin with the $(\Longrightarrow)$ direction. After replacing the $y_i$ with multiples, we may assume $z$ takes the form $z = ap^{\epsilon} + y_1 + \cdots + y_s$ for some $a \in S$, and we now look for $s \in S$ satisfying $z^{1/p} = sp^{\epsilon / p} + y_1^{1/p} + \cdots + y_s^{1/p}$ or, in other words,
	\begin{align*}
		ap^\epsilon + y_1 + \cdots + y_s & = (sp^{\epsilon / p} + y_1^{1/p} + \cdots + y_s^{1/p})^p \\
			& = s^p p^\epsilon + p \big( \cdots \big) + y_1 + \cdots + y_s,
	\end{align*}
	which once again turns into a monic equation on $s$ after cancellation and division by $p^\epsilon$. The $(\Longleftarrow)$ direction also follows from a direct argument, as in part (i).

	The proof of statement (iii) is identical to that of (iv), so let us prove (iv). The assumptions on $\gamma$ mean that we can write $\gamma = c_1 p^{-1} + c_2 p^{-2} + \cdots + c_e p^{-e}$ for some integers $c_i$ with $0 \leq c_i \leq p-1$. Using part (ii), we then obtain

	\begin{align*}
		z^\gamma = \prod_{i = 1}^e (z^{1/p^i})^{c_i} \in \prod_{i = 1}^e (p^{\epsilon / p^i}, y_1^{1/p^i}, \dots , y_s^{1/p^i})^{c_i}
	\end{align*}
	Note that the $i$-th factor in the above product of ideals is generated by the monomials
	$$p^{\epsilon a_{i0}/ p^i} y_1^{a_{i1} / p^i} \cdots y_s^{a_{is} / p^i}$$
	as the $a_{ij}$ range through all nonnegative integers with $\sum_j a_{ij} = c_i$. We conclude that the product is generated by elements of the form
	$$p^{\epsilon \sum a_{0i} / p^{i}} \ y_1^{ \ \sum a_{i1} / p^i} y_2^{\ \sum a_{i2} / p^i} \cdots \ y_s^{\ \sum a_{is} / p^i}, $$
	and by setting $\alpha_j = \sum_i a_{ij} / p^i$, we see that all of these are contained in the ideal given.
\end{proof}

\begin{remark} 
In statement (iii) and (iv) above, the proof gives an even stronger statement. For example, in (iv), whenever we have $z \in (p^\epsilon, y_1, \dots , y_s)$ we can conclude that $z^\gamma$ belongs to the ideal generated by the monomials $p^{\epsilon \alpha_0} y_1^{\alpha_1} \cdots y_s^{\alpha_s}$ as above, where we impose the extra condition that the $\alpha_i$ add up to $\gamma$ with no carries in base $p$.
\end{remark}

Before our next remark, we remind the reader that if $R$ is an $F$-finite regular domain, a pair $(R, f^t)$ is \emph{$F$-regular} if there exists $e > 0$ such that 
\[
    R \xrightarrow{1 \mapsto F^e_* f^{\lceil tp^e \rceil}} F^e_* R
\]
splits.  This is equivalent to requiring that $\tau(R, f^t) = R$.  See also  \cite{HaraYoshidaGeneralizationOfTightClosure,TakagiInversion,BlickleMustataSmithDiscretenessAndRationalityOfFThresholds} noting we do not have an extra test element contribution since $R$ is regular.  

\begin{lemma}
\label{lem.FPTCoverEqualityWithRamification}
	Let $k$ be a perfect field of characteristic $p > 0$, and consider $R \ceq k\llbracket u,x \rrbracket$ a formal power series ring in two variables over $k$. Given integers $a, b \geq 2$ , an integer $\ell \geq 1$, and a real number $t > 0$, we have
	$$(R, u^{1-1/a} x^{1 - 1/b} (u^\ell +x^\ell )^t) \text{ is $F$-regular} \iff (R, (u^{a\ell} + x^{b \ell})^t) \text{ is $F$-regular.}$$
\end{lemma}
We give two proofs.  First, we point out this is a consequence of general theorems about the behavior of $F$-singularities under finite maps. For the second, we do a down-to-earth computation.
\begin{proof}[Proof \#1]
    Consider the extension $A = k\llbracket u', x'\rrbracket = k\llbracket u^a, x^b\rrbracket \subseteq k\llbracket u,x\rrbracket = R$ with induced map $g$ on $\Spec$.  We know that $(R, u^{1-1/a} x^{1 - 1/b} (u^\ell +x^\ell )^t)$ is strongly $F$-regular if and only if
    \begin{equation}
        \label{eq.ExtensionWithRamLikeDivisor}
        \Big(A, {a - 1 \over a} \Div(u') + {b-1 \over b} \Div(x'),  (u'^\ell + x'^{\ell})^t\Big)
    \end{equation}
    is.  Consider the map $\bT : k[u,x] \to k[u',x']$ sending $1$ to $1$ and the other monomial basis elements to zero.  The associated ramification-like-divisor is $\sR_{\bT} := (a-1)\Div(u) + (b-1)\Div(x)$ as in \cite{SchwedeTuckerTestIdealFiniteMaps}.  As 
    \[
        g^*\big({a - 1 \over a} \Div(u') + {b-1 \over b}\Div(x')\big) - \sR_{\bT}  =0 
    \]
    by \cite[Theorem 6.25]{SchwedeTuckerTestIdealFiniteMaps}, we see that
    \[ 
        \bT \big( \tau(R, (u^{a\ell} + x^{b \ell})^t) \big)= \tau\big(A, {a - 1 \over a} \Div(u') + {b-1 \over b} \Div(x'),  (u'^\ell + x'^{\ell})^t\big).
    \] 
    Hence, $(R, (u^{a\ell} + x^{b \ell})^t)$ being strongly $F$-regular implies that so is \autoref{eq.ExtensionWithRamLikeDivisor}, as $\bT$ surjects.  The converse follows via an argument of \cite{CarvajalRojasFiniteTorsors} as $\bT$ sends the maximal ideal of $R$ into the maximal ideal of $A$.  Indeed, if $\tau(R, (u^{a\ell} + x^{b \ell})^t)$ is contained in the maximal ideal of $R$, then its image is in the maximal ideal of $A$.
\end{proof}

\begin{proof}[Proof \#2]
	We first claim that the pair on the left is $F$-regular if and only if one has 
	\begin{equation} \label{eq-left-Freg}
	(u^{\ell} + x^{\ell})^{\lceil t p^e \rceil} \notin \big( u^{\lfloor p^e / a \rfloor}, x^{\lfloor p^e / b \rfloor} \big)
	\end{equation}
	for some integer $e \geq 0$. Indeed, since we work over a regular ring, the pair is $F$-regular precisely when the map $R \to F^e_* R$ given by $F^e_* u^{\lceil p^e(1 - 1/a) \rceil} x^{\lceil p^e(1 - 1/b) \rceil} (u^\ell + x^\ell)^{\lceil t p^e \rceil}$ splits or, equivalently, when
	$$u^{\lceil p^e(1 - 1/a) \rceil} x^{\lceil p^e(1 - 1/b) \rceil} (u^\ell + x^\ell)^{\lceil t p^e  \rceil} \notin (u^{p^e}, x^{p^e})$$
	for some integer $e \geq 0$. The claim then follows from the fact that $(u, x)$ is a regular sequence.

	Now let $\tilde R \ceq k \llb \tilde u, \tilde x \rrb$ be a new power series ring, and consider the continuous $k$-algebra homomorphism $\phi: R \to \tilde R$ given by $\phi(u) =  \tilde u^a$ and $\phi(x) = \tilde x^b$. Note that $\phi(u^\ell + x^\ell) = \tilde u^{a \ell} + \tilde x^{b \ell}$ and that for every integer $e \geq 0$ we have
	\begin{equation} \label{eq-phi-inverse}
		(u^{\lceil p^e / a \rceil}, x^{\lceil p^e/ b\rceil}) = \phi^{-1}\big( (\tilde u^{p^e}, \tilde x^{p^e}) \big).
	\end{equation}
	Therefore, when the pair on the left is $F$-regular,  \autoref{eq-left-Freg} and \autoref{eq-phi-inverse} give $(u^\ell + x^\ell)^{\lceil t p^e \rceil} \notin \phi^{-1} (\tilde u^{p^e}, x^{p^e})$ for some $e \geq 0$ or, equivalently, $(\tilde u^{a \ell} + \tilde x^{b \ell})^{\lceil t p^e \rceil} \notin (\tilde u^{p^e}, \tilde x^{p^e})$, thus giving that the pair on the right is $F$-regular.

	Now assume that the pair on the right is $F$-regular. By introducing the perturbation $c \ceq \tilde u^a \tilde x^b = \phi(ux)$, we know there is some integer $e \geq 0$ such that
	\[
        \phi \big( ux \cdot (u^\ell + x^\ell)^{\lceil t p^e \rceil} \big) = c \cdot (\tilde u^{a \ell} + \tilde x^{b \ell})^{\lceil t p^e \rceil} \notin (\tilde u^{p^e}, \tilde x^{p^e})
    \]
    or, equivalently, $ux \cdot (u^\ell + x^\ell)^{\lceil tp^e \rceil} \notin \phi^{-1} (\tilde u^{p^e}, \tilde x^{p^e})$. With \autoref{eq-phi-inverse}, we get $ux \cdot (u^\ell + x^\ell)^{\lceil tp^e \rceil} \notin (u^{\lceil p^e / a \rceil}, x^{\lceil p^e/ b\rceil}) $ and using once again that $(u, x)$ is a regular sequence, we conclude that 
	$$(u^\ell + x^\ell)^{\lceil tp^e \rceil} \notin (u^{\lceil p^e / a \rceil - 1}, x^{\lceil p^e / b \rceil - 1}).$$
    But this implies condition \eqref{eq-left-Freg}  and so we see that  that the pair on the left is $F$-regular by our first claim.
\end{proof}

\begin{remark}
     Proof \#1 also works in mixed characteristic for a slightly different statement.  It shows that $A = \bZ_p\llbracket x\rrbracket$ and $R = A[p^{1 \over a}, x^{1 \over b}]$ then
    \[
        \text{$(A, p^{1-1/a}x^{1-1/b}(p^{\ell} + x^{\ell})^t)$ is $+$-(or BCM-)regular } \Longleftrightarrow \text{ $(R, ((p^{{1 \over a}})^{a\ell} + (x^{1 \over b})^{b\ell})^t)$ is $+$-(or BCM-)regular.}
    \]
    The point is we have transformation rules for $+$-test ideals and BCM-test ideals by \cite{MaSchwedeSingularitiesMixedCharBCM}.
\end{remark}

We now study what a particular blowup of our singularity looks like.

\begin{lemma}
    \label{lem.StructureOfToricBlowup}
	Suppose $(R, \frm = (u,x), k)$ is a complete regular 2-dimensional local ring and consider $f = u^{\ell a} + x^{\ell b} \in R$ where $a$ and $b$ are relatively prime.  Let $\pi : X \to \Spec R$ be the normalized blowup of the ideal $(u^a, x^b)$.  Set $D$ to be the strict transform of $\Div(f)$ on $X$.  Then we have the following.  

	\begin{enumerate}
		\item $X$ has a unique reduced exceptional divisor $E \cong \bP^1_k$.\label{lem.StructureOfToricBlowup.ExceptionalIsP1}
		\item $-E$ is a $\pi$-ample Weil divisor.\label{lem.StructureOfToricBlowup.ExceptionalIsAntiample}
        \item $K_{X / R} =  (a+b-1)E$ and $\Div(f)$ pulls back to $D +(\ell ab) E$. 
 \label{lem.StructureOfToricBlowup.RelativeCanonical}
            \item The different of $K_X + E$ restricted to $E$ is $(1-{1 \over a})P_1 + (1-{1 \over b})P_2$ for some points $P_1, P_2$ on $E$.  \label{lem.StructureOfToricBlowup.DifferentComputation}

	    \item $D$ intersects $E$ at nonsingular points of $X$ (in particular, away from $P_1$ and $P_2)$.  Explicitly, if we consider the homogeneous coordinate ring $k[y,z]$ for $E$ with $V(y) = P_1$ and $V(z) = P_2$ ($0$ and $\infty$), then we have that $D$ restricted to $E$ is given by the equation $(y^\ell + z^{\ell})$. 
    \label{lem.StructureOfToricBlowup.StrictTransformIntersection}
    
\end{enumerate}
\end{lemma}
This lemma shows that the normalized blowup of $(u^a,x^b)$ behaves exactly as one expects from the toric setting.  Those willing to believe that may wish to skip the proof on first reading.
\begin{proof}
    Consider the finite extension $S = R[u^{1 \over b}, x^{1 \over a}]$ of $R$ (with maximal ideal $\frn \subseteq S$) and consider the map $\phi : S \to R$ sending $1$ to $1$ and the other natural basis elements $u^{i/b}x^{j/a} \mapsto 0$ for $0 \leq i < b$ and $0 \leq j < a$.  Consider the maximal ideal $\frn = (u^{1/b}, x^{1/a})S \subseteq S$ and notice that 
    \[ 
        \frn^{ab} \cap R = \phi(\frn^{ab}) = (u^i x^j\;|\; i,j \geq 0, ib + ja \geq ab) =: J.
    \]
    As $\frn^{ab}$ is a valuation ideal (associated to blowing up the origin), we see that so is $J \subseteq R$.  In particular, $J$ is integrally closed.
    As $R$ is regular and 2-dimensional, $J^n$ is also integrally closed for all $n$ by a result of Zariski, see \cite[Theorem 14.4.4]{HunekeSwansonIntegralClosure}.  Let $X$ denote the blowup of $J$, and $Y$ the blowup of $\frn^{ab}$ (which is the same as the blowup of $\frn$, although the induced relatively very ample divisors have different coefficients), and there is a finite map $\beta : Y \to X$ induced by the inclusion of Rees algebras $R[Jt] \subseteq S[\frn^{ab}t]$.  Note $\phi$ extends to these Rees algebras and hence also induces $\phi : f_* \cO_Y \to \cO_X$.  It easily follows that $Y$ is smooth, $X$ is normal and Cohen-Macaulay (it is a summand of $f_* \cO_Y$), and since $Y$ has a unique reduced exceptional divisor $E_Y$, we have that $X$ has a unique reduced exceptional divisor $E$.  We have the following diagram.
    \[
        \xymatrix{
        & E_Y \ar@{_{(}->}[d] \ar[r]^{\gamma} & E \ar@{_{(}->}[d]\\
             & Y \ar[d]_{\pi_Y} \ar[r]^{\beta} & X \ar[d]^{\pi}   & \\
            & \Spec S \ar[r]_{\alpha} & \Spec R
        }
    \]

    On $K(S) = K(Y)$, the exceptional divisor induces a discrete valuation $v : K(S)^{\times} \to \bZ$ which takes $u^{1 \over b} \mapsto {1}, x^{1 \over a} \mapsto {1}$, and sends units to $0$ (it can also be viewed as the $\frn$-adic order on $S$).  Consider a monomial $g = u^{i}x^j = (u^{1 \over b})^{ib} (x^{1 \over a})^{ja} \in R$.  Notice that $v(g) = ib + ja$.  Therefore, as $a, b$ are relatively prime, there exist monomials $h \in K(R)$ with $v(h) = 1$.  

    We consider the charts of $X$ and $Y$ corresponding to $u^a$ and $x^b$, which we label $X_u$, $Y_u$, $X_x$, and $Y_x$ respectively.  As it is easy to see that $Y_u = \Spec S[x^{1\over a}/u^{1\over b}] = \Spec S[u^{i\over b} x^{j\over a} \;|\; i+j \geq 0, j\geq 0] = \Spec S^u$, one applies $\phi$ and sees that $X_u = \Spec  R[u^{i}x^{j} \;|\; ib+ ja \geq 0, j \geq 0] = \Spec R^u$.  Similarly $X_x = \Spec R[u^ix^j \;|\; ib + ja \geq 0, i \geq 0] = \Spec R^x$.  

    On $Y_u$, the exceptional divisor $E_Y$ is defined by the principal ideal $(u^{1 \over b})$.       
 As $x^{1 \over a} / u^{1 \over b}$ is in the ring $S^u$, $(u^{1 \over b}) S^u$ is also generated as a module over $R$ by the monomials $(u^{i \over b} x^{j \over a} \;|\; i+j > 0)S$ (and the analogous result holds for $Y_x$).  Hence, applying $\phi$ on $X_u$ (respectively $X_x$) or using the valuation $v$, the ideal defining $E$ is generated by the following monomials, even using scalars from $R$, 
    \[ 
        I_{E,u} = (u^ix^j \in R^u \;|\; bi + aj > 0)R \;\;\;(\text{respectively }\; I_{E,x} = (u^ix^j \in R^x \;|\; bi + aj > 0)R).
    \] 
    A direct computation now easily implies that the exceptional divisor $E$ is isomorphic to $\bP^1_k$. 
 Indeed, $x^b/u^a \notin I_{E,u}$, and it is not difficult to see that $R^u/I_{E,u} \cong k[x^b/u^a]$, while $R^x/I_{E,x} \cong k[u^a/x^b]$. By construction, $E$ is an anti-ample Weil divisor.  In particular, we see that \ref{lem.StructureOfToricBlowup.ExceptionalIsP1} and \ref{lem.StructureOfToricBlowup.ExceptionalIsAntiample} hold.  

    We observed above that the valuation ring $T \subseteq K(R)$ associated to $E$ (or equivalently to $v|_{K(R)}$) has elements $h$ of value $1$.  Set $T_Y \subseteq K(S)$ to be the valuation ring associated to $v$ (equivalently, associated to $E_Y$).  Hence $T \subseteq T_Y$ have a common monomial uniformizer $h$ but has an extension of residue fields.  Note $\phi : T_Y \to T$ sends monomials of $T$ to themselves and is nonzero on the associated residue field extension.  We see thus that the divisor associated to $\phi \in \Hom_T(T', T) = \omega_{T'/T}$ has coefficient zero along $E$. 
    
    We now compute canonical divisors.  If we set $K_{R} = -\Div(u) - \Div(x)$, then associated to the map $\phi \in \sHom( S, R)$ we obtain an effective divisor $K_{S/R} = {(b-1)} \Div(u^{1 \over b}) + {(a-1)} \Div(x^{1 \over a}) = {b-1 \over b} \Div(u) + {a-1 \over a} \Div(x)$ (this is the ramification divisor if $p$ is relatively prime to $a$ and $b$, otherwise it can be thought of as a replacement for the ramification divisor, see \cite{SchwedeTuckerTestIdealFiniteMaps} for related discussion).  We can then set $K_{S} = \alpha^* K_{ R} + K_{S/R} = -\Div(u^{1 \over b}) - \Div(x^{1 \over a})$.  
    Associated to the blowup $\pi_Y$, we see that $K_Y = -E_Y - D_1 - D_2$ where $D_1$ and $D_2$ are the strict transforms of $\Div(u^{1 \over b})$ and $\Div(x^{1 \over a})$.   Viewing $\phi \in \sHom(f_* \cO_Y, \cO_X)$, we then obtain that $K_{Y/X} = {(b -1)} D_1 + {(a-1)} D_2 + 0E_Y$ where coefficient of $E$ was computed above.  From the formula $K_Y = K_{Y/X} + \beta^* K_X$ we see that $K_X = -E - D_x - D_u$ where $D_x$ and $D_y$ are the strict transforms of $\Div(x)$ and $\Div(u)$.  This is what is expected based on the toric setting.

    For \ref{lem.StructureOfToricBlowup.RelativeCanonical}, we compute the relative canonical  and the pullback of $\Div(f)$.  Observe that 
    \[
        \begin{array}{rl}
            \beta^* K_{X / R} = & \beta^*\big( K_X - \pi^* (-\Div(ux)) \big)\\
            = & \beta^*(-E - D_x - D_u) + \pi_Y^*(a \Div(u^{1 \over a}) + b \Div(b^{1 \over b}))\\
            = & -E_Y - a D_1 - b D_2 + {a+b}E_Y - a D_1 - b D_2\\
            = & (a + b - 1) E_Y            
        \end{array}
    \]
    As $\beta^* E = E_Y$, this implies that $K_{X/R} = (a+b-1)E$.  Next, we compute the pullback of $\Div(f)$.  By pulling back to $Y$ and employing a similar argument, we see that $\pi^* \Div(f) = (\ell ab)E + D$.  

     We can now compute the different.  Note, $K_X+E$ has coefficient zero along $E$ and so it suffices to compute $(K_X + E)|_E \sim_{\bQ} K_E + \mathrm{Diff}$ in a specific way.  We do our computation on the chart $U := X_u$ as the other chart is similar.  First, we can pick a section $s = x^b/u^a \in \cO_U(K_X+E)$.  This produces a divisor $D_s = (b-1)D_x|_U \sim (K_X + E)|_U$ on this chart.  Consider the image of this section $\overline{s}$ under the canonical map $\omega_U(E) \twoheadrightarrow \omega_E$.  Viewing these as ideal sheaves, we see that while $s$ does not generate $\omega_U(E) = \cO_U(-D_x)$, its image generates the image of $\omega_E$ of $\omega_U(E)$, which is the ideal $(x^b/u^a) \in k[x^b/u^a]$.
     Therefore, we see that the associated divisor $D_{\overline{s}} = 0 \sim K_E$.  
     The different is thus computed as 
     \[
        (D_s)|_E = D_{\overline{s}} + {\mathrm{Diff}}.
     \]
     As $D_s = (b-1)D_x|_U = {b - 1 \over b} \Div(x^b/u^a)$ we see that the different $\Diff$ is ${b - 1 \over b} P_1$ where $P_1 = D_1 \cap E$ (the origin in this chart). The analogous computation shows that on the other chart, the different is ${a -1 \over a}P_2$.  This proves \ref{lem.StructureOfToricBlowup.DifferentComputation}.

     For the final statement \ref{lem.StructureOfToricBlowup.StrictTransformIntersection}, we simply need to understand how $D$, the strict transform of $\Div(f)$, restricts to $E \cong \bP^1_k$.  We again work on the chart $U = X_u$ and observe $f = u^{\ell a} + x^{\ell b} = u^{\ell a} \big(1 + ({x^{b} \over u^{a}})^{\ell}\big)$.  As $1 + ({x^{b} \over u^{a}})^{\ell}$ does not vanish along the exceptional divisor, we see it must be the equation of the strict transform.  As the analogous result holds on the chart $X_x$, the result follows.  
\end{proof}

\begin{proposition}
\label{prop.fpt<=+pt<=lct}
	Let $R \ceq \Z_p \llb x \rrb$, let $a, b \geq 2$ be  integers, and consider the polynomial $f \ceq p^{a} + x^{b} \in R$.  Consider $R_0 \ceq \F_p \llb y, x \rrb$ and $f_0 \ceq y^{a} + x^{b} \in R_0$. Then we have
	$$\fpt(R_0, f_0) \leq \ppt(R, f) \leq {1 \over a} + {1 \over b}.$$
\end{proposition}
\begin{proof}
We may write $a = \ell A$ and $b = \ell B$, where $\ell=\gcd(a,b)$, so that $A$ and $B$ are relatively prime.  Consider the normalized blowup $\pi : X \to \Spec R$ of $(p^A, x^B)$ as in \autoref{lem.StructureOfToricBlowup}.  Note, every $+$-regular pair is automatically KLT  by \cite[Theorem 6.21]{MaSchwedeSingularitiesMixedCharBCM} ($\widehat{R^+}$ is big enough for that containment), or use Matlis duality and \cite[Definition 4.16]{BMPSTWW1}.  Hence, the computation of \autoref{lem.StructureOfToricBlowup} \ref{lem.StructureOfToricBlowup.RelativeCanonical} implies that $$\ppt(f) \leq \sup\{ t \;|\; A + B - 1 - t \ell AB > -1 \} = {1 \over \ell A} + {1 \over \ell B} = {1 \over a} + {1 \over b}$$ as desired. 
 Indeed, it is not hard to see that this is the log canonical threshold.

    For the other inequality, let $E$ be the exceptional divisor of $\pi$ and let $D$ denote the strict transform of $\Div(f)$.  Suppose now that $t < \fpt(f_0)$.  By \autoref{lem.FPTCoverEqualityWithRamification}, we see that $t < \fpt(R_0, y^{1-1/A}x^{1-1/B}(u^\ell + x^{\ell})^t)$.  Using  \autoref{lem.StructureOfToricBlowup} \ref{lem.StructureOfToricBlowup.DifferentComputation} and \ref{lem.StructureOfToricBlowup.StrictTransformIntersection}, we see that $(E, \Diff_E(tD))$ has a section ring whose associated pair is $(R_0, y^{1-1/A}x^{1-1/B}(u^\ell + x^{\ell})^t)$.  In particular, $(E, (K_X + E + tD)|_E)$ is globally $F$-regular and so \cite[Lemma 7.2]{MaSchwedeTuckerWaldronWitaszekAdjoint} implies that $(R, f^t)$ is $+$-regular (this also follows from \cite[Theorem 7.2]{BMPSTWW1}).  The result follows.
\end{proof}

\begin{remark}
    One can also use direct computation to prove that $\ppt(f) \leq {1 \over a} + {1 \over b}$.  We sketch this here.

    Let $\zeta \in R^+$ be a primitive $(2ab)$-th root of unity, let $e \geq 0$ be an integer and $\gamma \in \frac{1}{p^e} \Z$ be a rational number whose denominator is $p^e$. We then have
	\begin{align*}
		f^\gamma = (p^{ab/b} + x^{ab/a})^\gamma & = \prod_{i = 0}^{ab - 1} (p^{1/b} - \zeta^i x^{1/a})^\gamma.
	\end{align*}
	Since for every $i = 0, \dots , ab-1$ we have $(p^{1/b} - \zeta^i x^{1/a}) \in (p^{1/b}, x^{1/a})$, by \autoref{lemma-key}, we conclude that
	$$f^\gamma \in \bigg( p^{\beta / b} x^{\alpha / a} \ \bigg| \  {{\alpha, \beta \in \frac{1}{p^e} \Z} \atop {\alpha + \beta = \gamma}}   \bigg)^{ab} \sq \bigg( p^{\beta / b} x^{\alpha / a} \ \bigg| \  {{\alpha, \beta \in \frac{1}{p^e} \Z} \atop {\alpha + \beta = ab \gamma}}   \bigg).$$
	Therefore, whenever $\gamma \geq (1/a) + (1/b)$ we have $ab \gamma \geq a + b$, and hence every choice of $\alpha, \beta$ with $\alpha + \beta = ab \gamma$ must have $\alpha \geq a$ or $\beta \geq b$, and hence $f^\gamma \in (p, x)R^+$.
 \end{remark}

 \begin{remark}
    Suppose $F \in \bZ[x,y]$ is homogeneous of degree $d$, $f := F(p,y) \in \bZ_p\llbracket y\rrbracket = R$, and $f_0 := F(x,y) \in \bF_p\llbracket x,y\rrbracket$. Let $X \to \Spec \bZ_p\llbracket y\rrbracket$ denote the blowup of the origin with exceptional divisor $E$.  On the chart $\Spec R[y/p]$ corresponding to $p$, $f = p^d F(1, y/p)$.  On $E = V(p)$, the strict transform defined by $F(1,y/p)$  restricts to a divisor defined by $f_0(1,x/y)$.  Likewise on the chart corresponding to $x$, $f = x^d F(1,p/x)$.  Restricting the strict transform defined by $F(1, p/x)$ to $E = V(x)$, we get $F(1,y/x)$.  Putting this together, we see that the strict transform of $\Div(f)$ intersects the exceptional $\bP^1$ in $V(f_0)$ (viewed as a homogeneous equation). 

    The argument of \autoref{prop.fpt<=+pt<=lct} then implies that 
    \[ 
        \fpt(R_0, f_0) \leq \ppt(R, f) \leq \min\{t \;|\; 1 - td > -1 \} = {2 \over d}.
    \]
    In fact, the argument is easier as there is no additional contribution to the different to worry about (that is, \autoref{lem.StructureOfToricBlowup} is unnecessary).

    As a consequence, we see that the extremal bounds of \cite{KadyrsizovaKenkelPageSinghSmithVraciuWitt.LowerBoundsExtremal} also apply in mixed characteristic.  Analogous computations can also be done in higher dimension.
 \end{remark}

\section{Cusp-like singularities}

As usual, $p > 0$ is a fixed prime number, and we let $R \ceq \Z_p \llb x \rrb$ be a power series ring in one variable over $\Z_p$. Given integers $a, b \geq 2$ we let $f \ceq p^a + x^b \in R$, and we let $f_0 \ceq t^a + x^b \in \F_p \llb t, x \rrb$. Recall that we have
$$\fpt(R_0, f_0) \leq \ppt(R, f) \leq \frac{1}{a} + \frac{1}{b}.$$
In particular, we have $\ppt(f) = (1/a) + (1/b)$ whenever $\fpt(f_0) = (1/a) + (1/b)$. We also recall that, by \cite{HernandezFPureThresholdOfBinomial}, whenever $\fpt(f_0) \neq (1/a) + (1/b)$ then the denominator of $\fpt(f_0)$ is a power of $p$; that is, there is some integer $e \geq 1$ such that $p^e \fpt(f_0) \in \Z$. 

Our general goal in this section is to show that that, when $\fpt(f_0)$ is not too far from $(1/a) + (1/b)$, then we have an equality $\fpt(f_0) = \ppt(f)$. We begin by introducing some notation to set up the technical result. As usual, we assume a prime $p > 0$ is fixed.

Given a real number $\lambda > 0$ and an integer $e \geq 1$, we let $\rho_e(\lambda)$ be the largest rational number whose denominator is $p^e$ with the property that $\rho_e(\lambda) < \lambda$. This can be written out precisely as
$$\rho_e(\lambda) \ceq \frac{\lceil \lambda p^e \rceil - 1}{p^e}.$$

\begin{theorem} \label{thm-m-containment}
Let $p > 0$ be a prime number and consider the ring $R \ceq \Z_p \llb x \rrb$. Fix integers $a, b \geq 2$ and consider $f \ceq p^a + x^b \in R$. Fix an integer $e \geq 1$ and a number $\gamma \in (1/p^e) \Z \cap (0,1)$, and assume that for some integer $c \geq a - \lfloor a / p \rfloor$ one has
$$\gamma > \frac{1}{c} \left( \rho_e \left(\frac{c}{a}\right) + \rho_e \left( \frac{c}{b} \right) \right).$$
Then $f^\gamma \in (p, x) R^+$.
\end{theorem}
\begin{proof}
Let $\zeta$ be a primitive $(2c)$-th root of unity, and note that
$$f^\gamma = \big( (p^{a/c})^c + (x^{b/c})^c)^{\gamma} = \prod_{i = 0}^{c-1} (p^{a/c} + \zeta^{2i+1} x^{b/c})^\gamma.$$
	For every $i = 0, \dots, c$, we have $p^{a/c} + \zeta^{2i + 1} x^{b/c} \in (p^{a/c}, x^{b/c})$ and, by our assumption on $c$, we know that $a / c \leq p / (p-1)$. We can therefore apply \autoref{lemma-key} to conclude that 
	$$ \big( p^{a/c} + \zeta^{2i + 1} x^{b/c} \big)^\gamma \in \bigg( p^{(a/c) \alpha} \ x^{(b/c) \beta} \ \bigg| \ {{\alpha, \beta \in \frac{1}{p^e} \Z \cap [0,1)} \atop {\alpha + \beta = \gamma  }} \bigg),$$ 
	and therefore $f^\gamma$ is contained in the $c$-fold product of these ideals; that is, $f^\gamma$ belongs to the ideal of $R^+$ generated by monomials of the form
$$p^{(a/c)(\alpha_1 + \cdots + \alpha_c)} \ x^{(b/c) (\beta_1 + \cdots + \beta_c)},$$
where $\alpha_1, \dots , \alpha_c, \beta_1, \dots , \beta_c \in (1/p^e) \Z \cap [0,1)$ satisfy $\alpha_i + \beta_i = \gamma$ for all $i = 1, \dots  c$. Now, if $f^\gamma \notin (p, x)R^+$, then there is a choice of such $\alpha_i$ and $\beta_i$ for which $\sum \alpha_i < c / a$ and $\sum \beta_i < c / b$. Since $\sum \alpha_i$ and $\sum \beta_i$ both have denominator $p^e$, this gives
$$c \gamma = \sum \alpha_i + \sum \beta_i \leq \rho_e \left( \frac{c}{a} \right) + \rho_e \left( \frac{c}{b} \right),$$
contradicting the assumption on $\gamma$.
\end{proof}

\begin{corollary} \label{cor-ppt-test}
Consider the ring $R_0 \ceq \F_p \llb y, x \rrb$ and the element $f_0 \ceq y^a + x^b \in R_0$. Assume $\fpt(R_0, f_0) \neq (1/a) + (1/b)$ and fix an integer $e \geq 1$ such that $p^e \fpt(R_0, f_0) \in \Z$. If there is some integer $c \geq a - \lfloor a / p \rfloor$ for which the inequality
$$\fpt(R_0, f_0) > \frac{1}{c} \left( \rho_e \left(\frac{c}{a}\right) + \rho_e \left( \frac{c}{b} \right) \right)$$
holds, then $\fpt(R_0, f_0) = \ppt(f)$. In particular, setting $c = a$, we get $\fpt(R_0, f_0) = \ppt(R, f)$ whenever
$$\fpt(R_0, f_0) \geq \frac{1}{a} + \frac{1}{b} - \frac{1}{a p^e}.$$
\end{corollary}

\begin{proof}
	Setting $\gamma = \fpt(R_0, f_0)$ in \autoref{thm-m-containment}, we conclude that $\fpt(R_0, f_0) \geq \ppt(R, f)$. The reverse inequality is given by \autoref{prop.fpt<=+pt<=lct}.
\end{proof}

It is shown in \cite{HernandezFPureThresholdOfBinomial} that, given integers $a, b \geq 2$, the $F$-pure threshold of $f_0 \ceq t^a + x^b \in \F_p \llb t, x \rrb$ can be computed from the base-$p$ expansions of $1/a$ and $1/b$. We discuss this algorithm in the case where $p > ab$.

For each integer $i \geq 0$, let $x_i \in \{1, \dots , a-1\}$ be the standard representative of $p^i$ in $(\Z / a)^\times$, so that the expansion of $1/a$ in base $p$ is
\begin{align*}
	\frac{1}{a} & = \frac{1}{p} \left( \frac{ p x_0 - x_1}{a} \right) + \frac{1}{p^2} \left( \frac{p x_1 - x_2}{a} \right) + \frac{1}{p^3} \left( \frac{p x_2 - x_3}{a} \right) + \cdots \\
	\intertext{Similarly, letting $y_i \in \{1, \dots , b-1 \}$ be the standard representative of $p^i$ in $(\Z / b)^\times$, the expansion of $1/b$ in base $p$ becomes} 
	\frac{1}{b} & = \frac{1}{p} \left( \frac{p y_0 - y_1}{b} \right) + \frac{1}{p^2} \left( \frac{p y_1 - y_2}{b} \right) + \frac{1}{p^3} \left( \frac{p y_2 - y_3}{b} \right) + \cdots.
\end{align*}

Now let $m$ be the smallest integer $i \geq 0$ for which $(x_i / a) + (y_i / b) > 1$, with the understanding that $m = \infty$ if this never occurs. Note that, because $x_0 = y_0 = 1$, we always have $m \geq 1$.

This $m$ can also be thought of as follows: the sum $(1/a) + (1/b)$ first carries when performed in base $p$ at the $(m+1)$-th digit. The $F$-pure threshold of $f_0$ is then obtained as a suitable truncation of the sum $(1/a) + (1/b)$ in base $p$; more precisely, we have
$$\fpt(f_0) = \frac{1}{p^m} +  \sum_{i = 1}^m \frac{1}{p^i} \left( \frac{x_{i-1} p - x_i}{a} + \frac{y_{i-1} p - y_i}{b} \right) = \frac{1}{a} + \frac{1}{b} - \frac{1}{p^m} \left( \frac{x_m}{a} + \frac{y_m}{b} - 1 \right),$$
with the understanding that this is $(1/a) + (1/b)$ if $m = \infty$. 

Crucially for us, given integers $a, b \geq 2$, there exist rational numbers $C_i \geq 0$ and integers $m_i \geq 1$, indexed by $i \in (\Z / ab)^\times$, such that for all $p > ab$ we have
$$\fpt(f_0) = \frac{1}{a} + \frac{1}{b} - \frac{C_i}{p^{m_i}} \text{ when } p \equiv i \text{ mod } ab.$$
By virtue of \autoref{cor-ppt-test}, whenever we have $C_i \leq 1/a$ we can conclude that $\ppt(f) = \fpt(f_0)$ for all primes $p > ab$ for which $p \equiv i \text{ mod } ab$. Since the computation of $\fpt(f_0)$ is implemented in the \texttt{FrobeniusThresholds} package for \texttt{Macaulay2} \cite{HernandezSchwedeTeixeiraWittFrobeniusThresholds}, we have an effective way to find primes for which we get $\ppt(f) = \fpt(f_0)$; the result of such a search for $2 \leq a, b \leq 5$ are displayed in Table \ref{table.intro} in the Introduction. Note that the full strength of \autoref{cor-ppt-test} is used for testing some small primes, and that the omission of a given prime from the table does not entail that $\ppt(f)$ differs from $\fpt(f_0)$---only that we do not know whether it is true.

\begin{remark} \label{remark-table-small-cases}
	In \autoref{sec.sporadic} we show that $\ppt(R, f) = \fpt(R_0, f_0)$ for some other cusp-like singularities for which the hypotheses of \autoref{cor-ppt-test} fail: we show that this is the case for $R \ceq \Z_2 \llb x \rrb$ and $f = 2^a + x^2 \in R$ for any $a \geq 2$.
\end{remark}

\section{Sporadic examples} \label{sec.sporadic}

In this section we explore some sporadic examples, where we make crucial use of the following lemma, which utilizes the fact that the $p$-adic completion of $R^+$ is Cohen Macaulay for a power series ring $R$ over the $p$-adic integers \cite{BhattAbsoluteIntegralClosure}.

\begin{lemma} \label{lemma-reg-seqce}
	Let $R \ceq \Z_p \llb x_1, \dots, x_n \rrb$ be a power series ring over the $p$-adic integers. Let $\alpha_0, \alpha_1, \dots , \alpha_n \in \Q \cap [0, 1)$, and let $h \in R^+$. Then
	$$p^{\alpha_0} x_1^{\alpha_1} \cdots x_n^{\alpha_n} h \in (p, x_1, \dots , x_n)R^+ \implies h \in (p^{1 - \alpha_0}, x_1^{1 - \alpha_1}, \dots , x_n^{1 - \alpha_n}) R^+.$$
\end{lemma}
\begin{proof}
    Fix $\gamma_i \in \bZ$ the denominator of $\alpha_i$.  Consider $S = R[p^{1/\gamma_0}, x_1^{1/\gamma_1}, \dots, x_n^{1/\gamma_n}]$ also a regular ring.
    Suppose $p^{\alpha_0} x_1^{\alpha_1} \cdots x_n^{\alpha_n} h \in (p, x_1, \dots , x_n)R^+$.  Then certainly $p^{\alpha_0} x_1^{\alpha_1} \cdots x_n^{\alpha_n} h \in (p, x_1, \dots , x_n)\widehat{R^+}$ where $\widehat{R^+}$ is the $p$-adic completion.  As $\widehat{R^+}$ is big Cohen-Macaulay over $S$  by \cite{BhattAbsoluteIntegralClosure}, we see that $p^{1/\gamma_0}, x_1^{1/\gamma_1}, \dots, x_n^{1/\gamma_n}$ is a regular sequence on $\widehat{R^+}$ and so
    \[
        h \in (p^{1 - \alpha_0}, x_1^{1 - \alpha_1}, \dots , x_n^{1 - \alpha_n}) \widehat{R^+}.
    \]
    Since the ideal on the right contains $p$, the result follows.
\end{proof}

\subsection{{The cusp $\mathbf{x^2 + y^3}$ in mixed characteristic 2}} We consider an honest binomial where $p$ does not replace one of the variables.

\begin{lemma}
\label{thm.CuspNoPChar2}
Let $R \ceq \Z_2 \llb x, y \rrb$ be a power series ring in two variables over the $2$-adic integers. Further suppose that $B$ is a big Cohen-Macaulay $R^+$ algebra.  For $f \ceq x^2 + y^3 \in R$, we have $ 2^{1/4} \, f^{1/2} \notin (2,x,y) B$. 
\end{lemma}

We originally wrote down the proof for $B = R^+$ utilizing \autoref{lemma-reg-seqce}.  However, Linquan Ma pointed out that our argument applies to an arbitrary big Cohen-Macaulay $B$ --- a stronger statement yielding \autoref{cor.CuspNoPChar2} below.

\begin{proof}
	Suppose $ 2^{1/4} \, f^{1/2} \in (2,x,y) B$ for a contradiction. Then as $B$ is big Cohen-Macaulay, we have $f^{1/2} \in (2^{3/4}, x,y)B$. Fix a $y^{1/2}$ and $2^{1/4}$ in $R^+$ and set $g \ceq x + y^{3/2}$. Then we have $f \equiv g^2 \bmod{(2) B} $. This implies that $(f^{1/2}  - g)^2 \in (2)B$, and \autoref{lemma-key} gives that $f^{1/2} - g \in (2^{1/2})B$. On the other hand, we have $f^{1/2} - g \in (2^{3/4}, x, y)$, and by using that $B$ is big Cohen-Macaulay again we conclude that 
    $$f^{1/2} - g \in (2^{1/2}) B \cap (2^{3/4}, x, y)B = (2^{3/4}, 2^{1/2} x, 2^{1/2} y) B.$$
    Now write $f^{1/2} = x + y^{3/2} + ax 2^{1/2} + by 2^{1/2} + c2^{3/4}$. Now, looking at the equation 
    $$(x + y^{3/2} + ax 2^{1/2} + by 2^{1/2} + c2^{3/4})^2 = x^2 + y^3$$
    modulo $2^{3/2}, x^2, y^2$, we conclude that $2 x y^{3/2} \in (2^{3/2}, x^2, y^2) B$, and using that $B$ is big Cohen-Macaulay yet again gives that $1 \in (2^{1/2}, x, y^{1/2}) B$, a contradiction.
\end{proof}

\begin{corollary}
\label{cor.CuspNoPChar2}
    With notation as in \autoref{thm.CuspNoPChar2}, we have that $\tau_+(R, f^{1/2 + \epsilon}) = R$ for all $1 \gg \epsilon > 0$.  In particular, $\ppt(R, f) > 1/2 = \fpt(\bF_p\llbracket x,y \rrbracket, f)$.
\end{corollary}
\begin{proof}
    Fix a sufficiently large perfectoid big Cohen-Macaulay $R^+$-algebra $B$ so that we have $\tau_B(R, f^t) = \tau_{\mathrm{BCM}}(R, f^{t})$ for all $t$ in the sense of \cite{BMPSTWW-RH, MaSchwedeSingularitiesMixedCharBCM}.  In particular, enlarging $B$ does not change this mixed characteristic test ideal.  As $R$ is regular, $B$ is faithfully flat over $R$ (see \cite[Page 77]{HochsterHunekeInfiniteIntegralExtensionsAndBigCM} or \cite[Lemma 2.9]{BhattAbsoluteIntegralClosure}).  Thus, arguing as in the proof of \autoref{lem.EquivPurityForCompletionIdealContainment} (in particular, the implications \ref{lem.EquivPurityForCompletionIdealContainment.IdealContainmentWithompletion} $\Rightarrow$ \ref{lem.EquivPurityForCompletionIdealContainment.PurityViaInjectiveHull} $\Rightarrow$ \ref{lem.EquivPurityForCompletionIdealContainment.PurityNoCompletion}), we see that since $f^{1/2} \notin (2,x,y)B$ by \autoref{thm.CuspNoPChar2}, that $R \xrightarrow{1 \mapsto f^{1/2}} B$ splits.  
     We thus see that $\tau_{\mathrm{BCM}}(R, f^{1/2}) = R$.
    It follows from \cite[Proposition 6.4]{MaSchwedePerfectoidTestideal} that $\tau_{\mathrm{BCM}}(R, f^{1/2 + \epsilon}) = R$ for all $1 \gg \epsilon > 0$.  But we always have that 
    \[
        \tau_{\mathrm{BCM}}(R, f^{1/2 + \epsilon}) \subseteq \tau_+(R, f^{1/2 + \epsilon})
    \]
    since $B$ is an $R^+$-algebra, and so the corollary follows.
\end{proof}

This suggests the following question.  

\begin{question}
    Is $\ppt(R, f) = 5/6$, or perhaps more interestingly, is it something strictly in between $1/2$ and $5/6$?
\end{question}

\subsection{3 lines}

We consider an equation similar to an arrangement of 3 lines, but where $p$ replaces a variable.

\begin{proposition}
\label{lem.3Lines}
    Let $p > 0$ be a prime number and consider the rings $R \ceq \Z_p \llb x \rrb$ and $R_0 \ceq \F_p \llb t, x \rrb$. Let $f=px(p-x)$ and $f_0=tx(t-x)$. Then we have  
    {\upshape
	$$\ppt(f) =\mathrm{fpt}(f_0) = \begin{cases} 
						2/3 \text{ if } p = 3 \\
						(2p - 1)/3p \text{ if } p \equiv 2 \mod 3 \\
						2/3 \text{ if } p \equiv 1 \mod 3.
				\end{cases}$$} 
\end{proposition}

\begin{proof}
By \autoref{prop.fpt<=+pt<=lct}, it suffices to show that $  \ppt(f) \le \mathrm{fpt}(f_0)$, i.e. $f^\alpha \in \mathfrak{m}R^+ $ for any $\alpha \ge \mathrm{fpt}(f_0)$. By \autoref{lemma-reg-seqce}, it suffices to show that
$$ (x-p)^\alpha \in (p^{1-\alpha},x^{1-\alpha}).$$
Suppose $\alpha=a/p^e$. Then by \autoref{lemma-key}, we know that 
$$ (x-p)^{1/p^e}\in (p^{1/p^e},x^{1/p^e}).$$
Hence $ (x-p)^{a/p^e}\in (p^{1/p^e},x^{1/p^e})^a$. When $p \equiv 2 \mod 3$, we just take $a=(2p-1)/3$ and $e=1$. Then by pigeonhole principle, we have $(p^{1/p},x^{1/p})^{(2p-1)/3} \subset (p^{(p+1)/3p},x^{(p+1)/3p})$ and hence the desired containment. Otherwise, when $a/p^e \ge 2/3$, we still have $(p^{1/p^e},x^{1/p^e})^a \subset (p^{(p^e-a)/p^e},x^{(p^e-a)/p^e})$ by pigeonhole principle.  

When $p \not\equiv 2 \mathrm{mod} 3$ then it is well known that $\fpt(f_0) = 2/3 = \lct(f)$, and so there is nothing to do.
\end{proof}

\subsection{The polynomial $\mathbf{2^a + x^2}$ in mixed characteristic 2.}

Let $R \ceq \Z_2 \llb x \rrb$ be a power series ring over the 2-adic integers, let $a \geq 2$ be an integer, and consider the polynomial $f \ceq 2^a + x^2 \in R$. As usual, we consider the equal characteristic $2$ analogues $R_0 \ceq \F_2 \llb y, x \rrb$ and $f_0 \ceq y^a + x^2 \in R_0$. The polynomial $f$ does not always satisfy the hypotheses of \autoref{cor-ppt-test} but, nonetheless, we show that $\ppt(R, f) = 1/2 = \fpt(R_0, f_0)$ with a more elementary argument.

\begin{proposition} \label{prop.2ax2}
	Let $R \ceq \Z_2 \llb x \rrb$, $a \geq 2$ be an integer, and $f \ceq 2^a + x^2 \in R$. Then $\ppt(R, f) = 1/2$.
\end{proposition}
\begin{proof}
	From the algorithm in \cite{HernandezFPureThresholdOfBinomial} discussed above, one sees that $\fpt(\F_2 \llb y, x \rrb, y^a + x^2) = 1/2$ for every $a \geq 2$. By \autoref{prop.fpt<=+pt<=lct}, it is then enough to show that $(2^a + x^2)^{1/2} \in (2, x) R^+$. To do so, it is enough to prove that there is some $s \in R^+$ such that
	$$(2^a + x^2)^{1/2} = 2 s + x.$$
	Squaring both sides and canceling, we see that this equation is satisfied if and only if $s^2 + sx - 2^{a-2} = 0$ which, being monic in $s$, has a solution in $R^+$.
\end{proof}

\section{Further questions}

There is a wide literature of computations of $F$-pure thresholds which would be natural to explore in mixed characteristic as identified in the introduction.  First, we would like to understand the examples $p^a + x^b$ in all mixed characteristics $(0,p)$. Here are the smallest examples for which we do not know the answer:

\begin{question}
    What is $\ppt(\Z_3 \llb x \rrb , x^3 + 27)$? What is $\ppt(\Z_2 \llb x \rrb, x^4 + 16)$?    
\end{question}

\begin{question}
	With notation as in this paper, for binomials $f_0 = y^a + x^b \in \bF_p\llbracket x,y\rrbracket$, and $f = p^a + x^b \in \Z_p\llbracket x \rrbracket$, are the only possibilities for the $\ppt$, either $\fpt(f_0)$ or the $\lct(f_0) = {1 \over a} + {1 \over b}$?
\end{question}

One can ask the analogous questions in higher dimensions (see \autoref{thm.CuspNoPChar2}), or for other equations (like line arrangements, see \autoref{lem.3Lines}).  Another interesting target would be elliptic curves and higher dimensional Calabi-Yau hypersurfaces, even diagonal ones such as $x^3 + y^3 + z^3 \in \bZ_p\llbracket x,y,z\rrbracket$ or $p^3 + y^3 + z^3 \in \bZ_p\llbracket y, z\rrbracket$, see also \cite{BhattSingh.FPTOfCalabiYau}.  

\bibliographystyle{skalpha}
\bibliography{MainBib.bib}
\end{document}